\documentclass[11pt,a4paper]{amsart}
\usepackage{amsfonts}

\usepackage{}
\usepackage[mathscr]{eucal}
\usepackage{amssymb}
\usepackage{amsbsy}

\usepackage{CJK,CJKnumb}
\usepackage[CJKbookmarks,colorlinks,
            linkcolor=black,
            anchorcolor=black,
            citecolor=blue]{hyperref}
\usepackage{color,soul}              
\usepackage{indentfirst}        
\usepackage{latexsym,bm}        
\usepackage{graphics,graphicx}
\usepackage{cases}
\usepackage{pifont}
\usepackage{txfonts}
\usepackage{xcolor}
\usepackage[all,knot,poly,cmtip]{xy}
\usepackage[usenames,dvipsnames]{pstricks}
\usepackage{epsfig}
\usepackage{pst-grad} 
\usepackage{pst-plot} 

\allowdisplaybreaks[4]  

\numberwithin{equation}{section}

\newcommand{\al}{\alpha}
\newcommand{\be}{\beta}

\newcommand{\emp}{\emptyset}
\newcommand{\ga}{\gamma}
\newcommand{\Ga}{\Gamma}
\newcommand{\la}{\lambda}

\newcommand{\La}{\Lambda}
\newcommand{\ot}{\otimes}

\newcommand{\Om}{\Omega}

\newcommand{\te}{\theta}

\newcommand{\ve}{\varepsilon}
\newcommand{\vt}{\vartheta}

\newcommand{\mc}{\mathscr{C}}
\newcommand{\pp}{\mathscr{P}}

\newcommand{\bN}{\mathbb{N}}
\newcommand{\bQ}{\mathbb{Q}}

\newcommand{\bZ}{\mathbb{Z}}

\newcommand{\mb}[1]{\mbox{#1}}
\newcommand{\mi}{\mbox{id}}

\newcommand{\bs}[1]{{\scriptsize\mbox{#1}}}

\newcommand{\lb}{\left(}
\newcommand{\rb}{\right)}
\newcommand{\rw}{\rightarrow}
\newcommand{\beq}{\begin{equation}}
\newcommand{\eeq}{\end{equation}}
\DeclareMathOperator{\ke}{Ker}
\DeclareMathOperator{\od}{o}
\DeclareMathOperator{\lex}{lex}
\DeclareMathOperator{\wll}{wll}
\DeclareMathOperator{\ly}{LYN}
\DeclareMathOperator{\el}{eLYN}

\begin{document}

\newtheorem{theorem}{Theorem}[section]

\newtheorem{lem}[theorem]{Lemma}

\newtheorem{cor}[theorem]{Corollary}
\newtheorem{prop}[theorem]{Proposition}

\theoremstyle{remark}
\newtheorem{rem}[theorem]{Remark}

\newtheorem{defn}[theorem]{Definition}

\newtheorem{exam}[theorem]{Example}

\theoremstyle{conjecture}
\newtheorem{con}[theorem]{Conjecture}

\renewcommand\arraystretch{1.2}

\title[Towards a polynomial basis of PQSym]{Towards a polynomial basis of the algebra of peak quasisymmetric functions}

\author[Li]{Yunnan Li}
\address{School of Mathematics and Information Science, Guangzhou University, Waihuan Road West 230, Guangzhou 510006, China}
\email{ynli@gzhu.edu.cn}

\date\today
\subjclass[2010]{Primary 05E05; Secondary 05A05, 13B25}

\begin{abstract}
Hazewinkel proved the Ditters conjecture that the algebra of quasisymmetric functions over the integers is free commutative by constructing a nice polynomial basis. In this paper we prove a structure theorem for the algebra of peak quasisymmetric functions (PQSym) over the integers. It provides a polynomial basis of PQSym over the rational field, different from Hsiao's basis, and implies the freeness of PQSym over its subring of symmetric functions spanned by Schur's Q-functions.
\end{abstract}

\keywords{peak quasisymmetric function, polynomial basis, lambda ring}

\maketitle


\section{Introduction}
As an important nonsymmetric generalization of symmetric functions, quasisymmetric functions were introduced by Gessel in \cite{Ge} to deal with the combinatorics of $P$-partitions. The algebra of quasisymmetric functions (QSym) has a natural Hopf algebra structure. In fact, quasisymmetric functions have been around since at least 1972, and at that time QSym appeared as the dual algebra of the Leibniz-Hopf algebra over the integers; see \cite{Di}. Perhaps even more importantly, the Leibniz-Hopf algebra is precisely the algebra of noncommutative symmetric functions (NSym), systematically studied in \cite{GKL} with subsequent papers.
And the graded Hopf duality between QSym and NSym was proved in \cite{MR}.

One long-standing conjecture due to Ditters in \cite{Di} states that the algebra of quasisymmetric functions over the integers is free polynomial. This is of great interest, for instance because of the role it plays in a classification theory for noncommutative formal groups. There were many meaningful attempts to solve the Ditters conjecture, and the first rigorous proof was given by Hazewinkel in \cite{Haz}, where he first dealt with a $p$-adic version of the Ditters conjecture then completed the case over the integers.
Later, another more direct proof appeared in \cite{Haz1} using the technique of lambda rings; see also \cite[16.71]{Haz2}, \cite[\S 6]{HNK}.
In fact, Hazewinkel gave a nice structure theorem for QSym, which constructs a polymonial basis of QSym and implies that QSym is free over its subring $\La$ of symmetric functions. We mention that Hazewinkel's structure theorem was also generalized in \cite{GX} to the one for a large class of mixable shuffle algebras arising from the construction of free Rota-Baxter algebras. In topology QSym arises as the cohomology of the loop space of the suspension of the infinite complex projective space, $\Om\Sigma CP^\infty$. It gives an alternative approach to the Ditters conjecture from algebraic topology based on James' splitting of $\Om\Sigma CP^\infty$; see \cite{BR}.

On the other hand, peak quasisymmetric functions were first considered by Stembridge in \cite{Ste} to develop Stanley's theory of $P$-partitions to the enriched case. As a Hopf subalgebra of QSym, the algebra of peak quasisymmetric functions, denoted PQSym, was widely studied and deeply related to many topics in combinatorics, geometry and representation theory; see \cite{BHT,BMSW,BHW}. In particular, PQSym has several nice bases refining classical Schur's Q-functions, one of which consists of Stembridge's peak functions defined as the weight enumerators of all enriched $P$-partitions of chains. Meanwhile, the graded Hopf dual of PQSym is precisely the peak algebra of symmetric groups (Peak) and such duality was studied in detail in \cite{Sch}.

In \cite{Hsi} Hsiao defined another nice basis of PQSym, called monomial peak functions and obtained by monomial quasisymmetric functions via Stembridge's descent-to-peak map. From this monomial-like basis, He proved that the peak algebra $\mb{Peak}_\bQ$ is isomorphic to the concatenation Hopf algebra over $\bQ$, whose coproduct is adjoint to the shuffle product. Hence, its Hopf dual $\mb{PQSym}_\bQ$ is a shuffle algebra over $\bQ$. A well-known theorem of Radford says that a shuffle algebra is freely generated by its subset of Lyndon words. In particular, Hsiao found the corresponding polynomial basis, containing all the Newton power sums $p_n$.

Inspired by the work of Hazewinkel and Hsiao, we give a structure theorem for PQSym in this paper (Theorem \ref{fr}), as a peak version of Hazewinkel's one for QSym in \cite{Haz1}. First we figure out a lambda ring structure on PQSym, then construct a set of polynomial generators of PQSym and find the complete relations they are subject to. In particular, we obtain another free polynomial basis of PQSym over $\bQ$, different from Hsiao's one.

In \cite{SY}, Savage and Yacobi proved the freeness of QSym over its subring $\La$ of symmetric functions, alternatively using the technique of representation theory, namely, Heisenberg doubles arising from the tower of 0-Hecke algebras. Later, in \cite{Li} we applied such method to the case of PQSym by the tower of 0-Hecke-Clifford algebras, in order to prove the freeness of PQSym over its subring $\Ga$ spanned by Schur's Q-functions. Consequently, it is natural to ask for a more straightforward approach. That intrigues us to give the structure theorem \ref{fr} here immediately implying such freeness. Moreover, by our structure theorem it only needs to handle the case of $\Ga$ to show that PQSym is free polynomial, but even for the subring $\Ga$ any result about this is hardly known so far. Besides, a topological interpretation for PQSym (in particular $\Ga$) along the line of that in \cite{BR} for QSym may be also interesting to consider.

The organization of the paper is as follows. In $\S 2$ we introduce some notations in combinatorics, definitions for QSym and PQSym,
the terminology of lambda rings, and the lambda ring structure of QSym. In $\S 3$ we first recall Hsiao's result about monomial peak functions, then prove that PQSym is a lambda quotient ring of QSym under Stembridge's decsent-to-peak map. After that, we eventually give our structure theorem for PQSym closely related to Hazewinkel's one for QSym.

\section{Background}

\subsection{Notations and definitions}
Denote by $\bN$ (resp. $\bN_{\od}$) the set of positive (resp. odd positive) integers. Given any $m,n\in\bN,\,m\leq n$, let $[m,n]:=\{m,m+1,\dots,n\}$ and $[n]:=[1,n]$ for short. Let $\mc(n)$ be the set of compositions of $n$, consisting of ordered tuples of positive integers summed up to $n$ and $\mc:=\bigcup\limits_{n\geq1}^.\mc(n)$. Write $\al\vDash n$ when $\al\in \mc(n)$. Given $\al=(\al_1,\dots,\al_r)\vDash n$, let its length $\ell(\al):=r$, its weight $\mb{wt}(\al):=n$ and define its associated \textit{descent set} as
\[D(\al):=\{\al_1,\al_1+\al_2,\dots,\al_1+\cdots+\al_{r-1}\}
\subseteq[n-1].\]
Also, the refining order $\leq$ on $\mc(n)$ is defined by
\[\al\leq\be\mb{ if and only if }D(\be)\subseteq D(\al),\,\forall\al,\be\vDash n.\]
Let $k\cdot\al:=(k\al_1,\dots,k\al_r)$ for any $k\in\bN$. Let $\al*\be$ be the concatenation of $\al$ and $\be$.

We highlight the subset $\mc_{\od}(n)$ of $\mc(n)$, consisting of compositions of $n$ with odd parts. Let $\mc_{\od}:=\bigcup\limits_{n\geq1}^.\mc_{\od}(n)$.
It is well-known that
\[|\mc_{\od}(n)|=f_{n-1},\]
where $\{f_n\}_{n\geq0}$ is the Fibonacci sequence defined recursively by
\[f_0=f_1=1,\,f_n=f_{n-1}+f_{n-2},\,n\geq2.\]
It corresponds to two possible cases for arbitrary $\al\in\mc_{\od}(n+1)$. Namely, the first part of $\al$ is 1 or not smaller than 3.

Throughout this paper, we only consider two base rings $\bZ$ and $\bQ$.
For any ring $A$,  we always denote $A_\bQ:=A\ot_\bZ\bQ$, when the base ring changes from $\bZ$ to $\bQ$.

\subsection{Peak quasisymmetric functions}

Let $\La$ be the graded ring of symmetric functions in the commuting variables $x_1,x_2,\dots$ with integer coefficients, then it has two usual polynomial basis, the \textit{elementary symmetric functions} $\{e_n:n\in\bN\}$ and the \textit{complete symmetric functions} $\{h_n:n\in\bN\}$. Namely,
\[\La=\bZ[e_n:n\in\bN]=\bZ[h_n:n\in\bN].\]
Let $\Ga$ be the subring of $\La$ with the generators $q_n\,(n\geq1)$ defined by
\beq\label{gcf}\sum_{n\geq0}q_nz^n=\prod_{i\geq1}\dfrac{1+x_iz}{1-x_iz}.\eeq
That is, $q_n=\sum_{i=0}^n h_ie_{n-i},\,n\in\bN$, and they satisfy the \textit{Euler relations}
\beq\label{eu} \sum_{i=0}^n(-1)^iq_iq_{n-i}=0,\,n\in\bN,\eeq
which are complete to define $\Ga$. Recall that a \textit{(strict) partition} is a composition with (strictly) decreasing parts, then  \[\{q_\la:\la\mb{ strict partition}\}\]
is a $\bZ$-basis of $\Ga$, where $q_\la:=q_{\la_1}q_{\la_2}\cdots$
and $q_\emptyset=1$ for $\la=\emptyset$ by convention. Besides, $\Ga$ has another $\bZ$-basis
\[\{Q_\la:\la\mb{ strict partition}\},\]
called \textit{Schur's Q-functions}. We also note that
\[\La_\bQ=\bQ[p_n:n\in\bN]\mb{ and }\Ga_\bQ=\bQ[p_n:n\in\bN_{\od}]
=\bQ[q_n:n\in\bN_{\od}],\]
where $p_n:=\sum_{i\geq1}x_i^n\,(n\in\bN)$, the \textit{Newton power sums}.
Moreover, there exists a surjective ring homomorphism
\[\te:\La\rw\Ga,\quad h_n\mapsto q_n,\,n\in\bN.\]
Then $\te(e_n)=q_n,\,\te(p_n)=(1-(-1)^n)p_n,\,n\in\bN$.
For the basics of the rings $\La$ and $\Ga$, one can refer to \cite[Ch. I, Ch. III, \S 8]{Mac}.

As a non-symmetric generalization of $\La$, the algebra of \textit{quasisymmetric functions} over the integers, denoted by QSym, is a subring of the power series ring $\bZ[[x_1,x_2,\dots]]$ in the commuting variables $x_1,x_2,\dots$ and has a $\bZ$-basis, the \textit{monomial quasisymmetric functions}, defined by
\[M_\al:=M_\al(x)=\sum\limits_{i_1<\cdots< i_r}x_{i_1}^{\al_1}\cdots x_{i_r}^{\al_r},\]
where $\al=(\al_1,\dots,\al_r)$ varies over the composition set $\mc$. The multiplication of $M_\al$ comes from the \textit{quasi-shuffle product} $\bowtie$ on $\bZ\mc$. Recall that such operation is recursively defined as follows:
\[\al\bowtie\emp=\emp\bowtie\al=\al,\,
\al\bowtie\be=(a_1)*(\al'\bowtie\be)+(b_1)*(\al\bowtie\be')+
(a_1+b_1)*(\al'\bowtie\be'),\]
where $\al=(a_1,\dots,a_r)=(a_1)*\al',\,\be=(b_1,\dots,b_r)=(b_1)*\be'$. Then by \cite[Lemma 3.3]{Eh}, we have
\beq\label{mo}
M_\al M_\be=\sum_{\ga\in\mc}c_\ga M_\ga,\,\al,\be\in\mc,
\eeq
when $\al\bowtie\be=\sum_\ga c_\ga \ga$ with $c_\ga\in\bZ$. There is another important $\bZ$-basis, the \textit{fundamental quasisymmetric functions}, defined by
\[F_\al:=F_\al(x)=\sum\limits_{i_1\leq\cdots\leq i_n\atop i_k<i_{k+1}\mb{ \tiny if }k\in D(\al)}x_{i_1}\cdots x_{i_n},\,\al\vDash n.\]
That means $F_\al=\sum_{\be\leq\al}M_\be$.

Now we introduce the algebra of \textit{peak quasisymmetric functions} over the integers defined in \cite{Ste} and denoted by PQSym. It is a subring of QSym. In order to define the usual bases of PQSym, we need the concept of peak subsets. Recall that $P$ is a \textit{peak subset} of $[n-1]$ if $P\subseteq[2,n-1]$ and $i\in P\Rightarrow i-1\notin P$. Denote by $\pp_n$ the collection of all peak subsets of $[n-1]$ and $\pp=\bigcup\limits^._{n\geq1}\pp_n$. Given $\al=(\al_1,\dots,\al_r)\vDash n$, let
\[P(\al):=\{i\in D(\al)\cap[2,n-1]:i-1\notin D(\al)\}\]
be its associated peak subset of $[n-1]$.

We need the following nice bijection between $\mc_{\od}(n)$ and $\pp_n$ for any  $n\in\bN$. Given $\al=(2i_1+1,\dots,2i_r+1)\in\mc_{\od}$, let
\[\hat{\al}:=(\overbrace{2,\dots,2}^{i_1},1,\overbrace{2,\dots,2}^{i_2},1,
\dots,\overbrace{2,\dots,2}^{i_r},1).\]
If $\hat{\al}=(\overbrace{1,\dots,1}^{j_1},2,\overbrace{1,\dots,1}^{j_2},2,
\dots,\overbrace{1,\dots,1}^{j_s},2,1,\dots,1)$ and set
\[S_\al:=\left\{\sum_{k=1}^l(j_k+2):l\in[s]\right\},\]
then the map $\al\mapsto S_\al$ gives a bijection between $\mc_{\od}(n)$ and $\pp_n$. Note that $S_\al=\emp$ when $\al=(1,1,\dots,1)$ and
\begin{align*}
&D(\al)=[n-1]\backslash(S_\al\cup(S_\al-1)),\\
&\ell(\al)=|\{1\mb{'s in }\hat{\al}\}|,\\
&|S_\al|=|\{2\mb{'s in }\hat{\al}\}|=\ell(\hat{\al})-\ell(\al)
=\dfrac{|\al|-\ell(\al)}{2}.
\end{align*}
For any $\al\vDash n$, we denote $\La(\al)$ the unique odd composition of $n$ such that $S_{\La(\al)}=P(\al)$ via such bijection.
For example, if $\al=(1,4,2,3)$, then $\be=\La(\al)=(1^3,5,1^2)$.
$\hat{\be}=(1^3,2^2,1^3),\,D(\be)=\{1,2,3,8,9\}$ and
$S_\be=\{5,7\}$.

One can easily check the following result by definition.
\begin{lem}\label{hat}
For any $\al,\be\in\mc_{\od}(n)$, $P(\hat{\al})=S_\al$, while $S_\al\subseteq S_\be$ if and only if $\hat{\al}\leq\hat{\be}$.
\end{lem}

Recall that Stembridge's \textit{peak functions} in PQSym can be defined by
\beq\label{ka}K_P=2^{|P|+1}\sum_{\al\vDash n\atop P\subseteq D(\al)\triangle(D(\al)+1)}F_\al,\,P\in\pp_n,\eeq
where $D\triangle(D+1)=D\backslash(D+1)\cup (D+1)\backslash D$ for any $D\subseteq[n-1]$ and $D+1:=\{x+1:x\in D\}$. Then $\{K_P\}_{P\in\pp_n}$ forms a $\bZ$-basis of PQSym and there also exists a surjective ring homomorphism
\[\vt:\mb{QSym}\rw\mb{PQSym},\quad F_\al\mapsto K_{P(\al)}.\]
Note that
\[K_P=\sum_{\al\vDash n\atop P\subseteq D(\al)\cup(D(\al)+1)}2^{\ell(\al)}M_\al,\,P\in\pp_n\]
and in particular,
\beq\label{ke}
K_{\emp_n}=q_n=2\sum_{\al\vDash n}F_\al=\sum_{\al\vDash n}2^{\ell(\al)}M_\al.\eeq
Here we write $K_\al:=K_{S_\al}$ for any $\al\in\mc_{\od}$.

On the other hand, $\La$ (resp. $\Ga$) is a subring of QSym (resp. PQSym) and the following commutative diagram holds:
\beq\label{sq}\xymatrix@=2em{\mb{QSym}\ar@{->}[r]^-{\vt}
&\mb{PQSym}\\
\La\ar@{->}[r]^-{\te}\ar@{->}[u]&\Ga\ar@{->}[u]},\eeq
where the vertical maps in the diagram are inclusions.

All the rings introduced above have nice Hopf algebra structures in combinatorics. We do not mention this aspect here and one can refer it in \cite{Eh,GKL,MR,Sch}, etc.

\subsection{QSym as lambda rings}
The concept of lambda rings first appeared in Grothendieck's work concerning the Riemann-Roch theorem. So far they have been widely studied in many areas, e.g. K-theory, representation theory of finite groups, the theory of free Lie algebras, etc.
Recalled that a \textit{lambda ring ($\la$-ring)} is a commutative ring $R$ equipped with extra operations
\[\la^i: R\rw R\]
that behave just like exterior powers (of vector spaces or representations). A \textit{morphism} of $\la$-rings $F:A\rw B$ is a morphism of rings that commutes with the exterior product operations, i.e. $F(\la^i_A(x))=\la^i_B(F(x)),\,x\in A$.

There are associated ring endomorphisms called \textit{Adams operations} (in algebraic topology) or \textit{power operations} on a lambda ring. Given a $\la$-ring $R$ with operations $\la^i:R\rw R$, define operations $\Psi^i:R\rw R$ by the formula
\beq\label{lp}\dfrac{d}{dt}\log\la_t(a)=\sum_{n=0}^{\infty}
(-1)^n\Psi^{n+1}(a)t^n,\eeq
where $\la_t(a)=1+\la^1(a)t+\la^2(a)t^2+\cdots$ for any $a\in R$. Set $\la^0(a)=1$ by convenience. Then a set of operations $\la^i:R\rw R$ for a torsion free ring $R$ turns it into a $\la$-ring if and only if the Adams operations $\Psi^i:R\rw R$ are all ring endomorphisms
and in addition satisfy
\beq\label{psi}\Psi^1=\mi,\,\Psi^m\circ\Psi^n=\Psi^{mn},\,m,n\in\bN.\eeq
Note also that the relation \eqref{lp} between the lambda operations and the Adams operations is precisely the same as that between the elementary symmetric functions and the Newton power sums. Thus there are the following useful determinantal formulas.
\begin{align}
&\label{lps}n!\la^n(a)=\begin{vmatrix}
\Psi^1(a)&1&0&\dots&0\\
\Psi^2(a)&\Psi^1(a)&2&\ddots&\vdots\\
\vdots&\vdots&\ddots&\ddots&0\\
\Psi^{n-1}(a)&\Psi^{n-2}(a)&\dots&\Psi^1(a)&n-1\\
\Psi^n(a)&\Psi^{n-1}(a)&\dots&\Psi^2(a)&\Psi^1(a)
\end{vmatrix},\\
&\label{pl}\Psi^n(a)=\begin{vmatrix}
\la^1(a)&1&0&\dots&0\\
2\la^2(a)&\la^1(a)&1&\ddots&\vdots\\
\vdots&\vdots&\ddots&\ddots&0\\
(n-1)\la^{n-1}(a)&\la^{n-2}(a)&\dots&\la^1(a)&1\\
n\la^n(a)&\la^{n-1}(a)&\dots&\la^2(a)&\la^1(a)
\end{vmatrix}.
\end{align}
For more about $\la$-rings, see \cite{Haz2, Kn}.

Recall that there is a simple lambda ring structure on $\bZ[[x_1, x_2,\dots]]$ given by
\[\la^i(x_j)=
\begin{cases}
x_j&\mb{if }i=1\\
0&\mb{if }i\geq2
\end{cases},\,j=1,2,\dots\]
of which the associated Adams endomorphisms are, obviously, the power operations
\[\Psi^n:x_j\mapsto x_j^n.\]

\begin{theorem}[{\cite[Th. 3.1]{Haz1}}]\label{qf}
QSym as a subring of $\bZ[[x_1,x_2,\dots]]$ is a lambda ring and the corresponding Adams operations are
\beq\label{pow}\Psi^n: M_α\mapsto M_{n\cdot\al},\,n\in\bN.\eeq
In particular, the ring $\La$ of symmetric functions is a lambda subring of QSym.
\end{theorem}

We also recall the \textit{plethysm} of symmetric functions. For $f,g\in\La$, write $g$ as a sum of monomials
\[g=\sum_{\al\in\mc}c_\al x^{\al}\]
and let
\[\prod_{i\geq1}(1+y_it)=\prod_{\al\in\mc}(1+x^\al t)^{c_\al}.\]
Define the plethysm $f\circ g\in\La$ by
\[f\circ g=f(y_1,y_2,\dots).\]
Clearly, the plethysm operation provides another kind of ring multiplication of $\La$.

Note that any $\la$-ring $R$ naturally has the following $\La$-module structure, when $\La$ is equipped with the plethysm product:
\beq\label{ac}\La\times R\rw R,(f,a)\mapsto f(a),\eeq
where
\beq\label{fa}f(a):=f(\Psi^1(a),\Psi^2(a),\dots)\eeq
is obtained by substituting $\Psi^n(a)$ for each $p_n$ in the polynomial form of $f$ in terms of the power sums (with rational coefficients). In order to see that \eqref{ac} is well-defined, one only needs to see that
\[f(a)=f(\la^1(a),\la^2(a),\dots),\]
which is derived from the polynomial form of $f$ in terms of the elementary symmetric functions instead (with integral coefficients) by substituting $\la^n(a)$ for each $e_n$. Hence, $f(a)$ certainly lies in $R$. Meanwhile, one can check that
\beq\label{co} f(g(a))=(f\circ g)(a),\eeq
thus \eqref{ac} really defines a $\La$-module action on $R$. Indeed, formula \eqref{co} comes from the following identity
\[\Psi^n(g)(a)=\Psi^n(g(a))=g(\Psi^n(a)),\]
which is obtained by the property of the Adam operations $\Psi^n$.

\section{The structure theorem for PQSym}
In this section, we give a structure theorem for PQSym using the technique of lambda rings, which is totally different from
that in Hsiao's paper \cite{Hsi}. First introduce \[L_\al:=\vt(M_\al),\,\al\in\mc_{\od}\]
and call them the \textit{monomial peak functions} as in \cite[\S 2]{Hsi}, while such $L_\al$ differs by a sign from the $\eta_\al$ defined there, that is, $\eta_\al=(-1)^{|S_\al|}L_\al$. In particular, $L_{(1^n)}=q_n,\,n\in\bN$.

Let $\mc_{\bs{e}}$ be the set of compositions with their last parts to be even. If $\al\notin\mc_{\bs{e}}$, then there is a unique factorization $\al=\al_{(1)}*\al_{(2)}*\cdots*\al_{(l)}$ such that the last part of each $\al_{(j)}$ is odd and all other parts are even. In this case, we abuse the notation to set
\[\vt(\al):=(|\al_{(1)}|,\dots,|\al_{(l)}|).\]
For example, if $\al=(1,2,2,1,2,3)$, then $\vt(\al)=(1,5,5)$. It is clear that $\vt(\al)=\al=\vt(\hat{\al})$ for any $\al\in\mc_{\od}$. Due to \cite[Th. 2.4]{Hsi}, we have the following result about the $L_\al$'s.
\begin{theorem}\label{bas}
The monomial peak functions $L_\al\,(\al\in\mc_{\od})$ form a $\bZ$-basis of PQSym and
\[\ke\,\vt=\mb{span}_\bZ\{M_\al-(-1)^{\ell(\al)+\ell(\be)}M_\be:
\al\in\mc_{\od},\vt(\be)=\al\}
\oplus\mb{span}_\bZ\{M_\al:\al\in\mc_{\bs{e}}\}.\]
\end{theorem}
In particular, $\vt(M_{\hat{\al}})=(-1)^{\ell(\al)+\ell(\hat{\al})}L_{\al},\,\al\in\mc_{\od}$.
Thus applying $\vt$ to the formula $F_{\hat{\al}}=\sum_{\hat{\be}\leq\hat{\al}}M_{\hat{\be}}$,
we have
\[K_\al=\sum_{\be\in\mc_{\od}(n)\atop S_\be\subseteq S_\al}(-1)^{|S_\be|}L_\be,\,\al\in\mc_{\od}(n),\]
by Lemma \ref{hat}. And $L_\al=\sum_{\be\in\mc_{\od}(n)\atop S_\be\subseteq S_\al}(-1)^{|S_\be|}K_\be$ by M\"{o}bius inversion formula.


Note that PQSym is not a lambda subring of QSym with the lambda ring structure given in Theorem \ref{qf}. However, we have the following result instead.
\begin{theorem}\label{la}
There exists a lambda ring structure on PQSym such that the map $\vt$ is a morphism of lambda rings. Its corresponding Adams operations are defined by
\beq\label{po}\Phi^n: L_\al\mapsto\begin{cases}
L_{n\cdot\al},&\mb{if }n\mb{ odd}\\
0,&\mb{if }n\mb{ even}
\end{cases},\,\al\in\mc_{\od},n\in\bN.\eeq
In particular, the ring $\Ga$ is a lambda subring of PQSym.
\end{theorem}
\begin{proof}
First by Theorem \ref{bas}, one can easily check that $\Psi^n(\ke\,\vt)\subseteq\ke\,\vt$ as $\vt(n\cdot\be)=n\cdot\vt(\be),\,n\in\bN_{\od},\be\notin\mc_{\bs{e}}$.
Hence, there exists a unique ring map $\Phi^n$ such that $\vt\circ\Psi^i$ factors through $\vt$ and $\Phi^n$. Namely, the following commutative diagram holds:
\[\xymatrix@=2em{\mb{QSym}\ar@{->}[r]^-{\vt}\ar@{->}[d]_-{\Psi^n}
&\mb{PQSym}\ar@{->}[d]^-{\Phi^n}\\
\mb{QSym}\ar@{->}[r]^-{\vt}&\mb{PQSym}}\]
which guarantees that $\Phi^n$ satisfies relation \eqref{psi} and becomes a set of Adams operations on PQSym, thus PQSym has the desired lambda ring structure such that $\vt$ becomes a morphism of lambda rings. Meanwhile,
\[\Phi^n(L_\al)=\vt\circ\Psi^n(M_\al)=\vt(M_{n\cdot\al})=
\begin{cases}
L_{n\cdot\al},&\mb{if }n\mb{ odd}\\
0,&\mb{if }n\mb{ even}
\end{cases},\,\al\in\mc_{\od}.\]
In particular, the ring $\Ga$ is a lambda subring of PQSym, due to \eqref{sq}. For any $n\in\bN_{\od}$, $L_{(n)}=2M_{(n)}=2p_n$, thus $\Phi^i(p_n)=\dfrac{1}{2}\vt\circ\Psi^i(M_{(n)})
=\dfrac{1-(-1)^i}{2}p_{in},\,i\in\bN$, in $\mb{PQSym}_\bQ$.
\end{proof}

By formula \eqref{lp} and \eqref{po}, we know that the lambda operations, denoted $\tilde{\lambda}^i\,(i\in\bN)$, on PQSym are given as follows.
\[\tilde{\la}_t(f)=\exp\lb\sum_{n\in\bN}
(-1)^{n-1}\dfrac{\Phi^n(f)t^n}{n}\rb=\exp\lb\sum_{n\in\bN_{\od}}
\dfrac{\Phi^n(f)t^n}{n}\rb,\,f\in\mb{PQSym}.\]
It implies that $\tilde{\la}_t(f)\tilde{\la}_{-t}(f)=1$, i.e.
\beq\label{eul}\sum_{i=0}^n(-1)^i\tilde{\lambda}^i(f)\tilde{\lambda}^{n-i}(f)=0,\eeq
of the same form as the Euler relations \eqref{eu}. In particular for $q_1=2p_1$, we have
\[\tilde{\la}_t(q_1)=\exp\lb\sum_{n\in\bN_{\od}}
\dfrac{\Phi^n(q_1)t^n}{n}\rb=\exp\lb\sum_{n\in\bN}\dfrac{1-(-1)^n}{n}p_nt^n\rb
=\prod_{i\geq1}\dfrac{1+x_it}{1-x_it}.\]
That means $\tilde{\la}_t(q_1)=\sum_{n\geq1}q_nt^n$ by \eqref{gcf}, thus
\beq\label{qa}\tilde{\la}^n(q_1)=q_n,\,n\in\bN.\eeq
It can also be obtained by the following commutative diagram as $\vt$ is a morphism of $\la$-rings.
\beq\label{com}\xymatrix@=2em{\mb{QSym}\ar@{->}[r]^-{\vt}\ar@{->}[d]_-{\la^n}
&\mb{PQSym}\ar@{->}[d]^-{\tilde{\la}^n}\\
\mb{QSym}\ar@{->}[r]^-{\vt}&\mb{PQSym}}\eeq

For now on we naturally interpret compositions as words over $\bN$. The \textit{lexicographic order} on these words is defined as follows. Let $\al=(a_1,\dots,a_r)$ and $\be=(b_1,\dots,b_s)$ be two words over $\bN$. Then $\al$ is said to be lexicographically equal
to or larger than $\be$, denoted $\al\geq_{\lex}\be$, if and only if there is a $j\in[\mb{min}\{r,s\}]$ such that
$a_1=b_1,\dots,a_{j-1}=b_{j-1}$ and $a_j>b_j$, or $r\geq s$ and $a_1=b_1,\dots,a_s=b_s$.

\begin{defn}
The \textit{proper tails} (suffixes) of a word
$\al=(a_1,a_2,\dots,a_m)$ are the words $(a_i,a_{i+1},\dots,a_m),i=2,3,\dots,m$. Words of length
1 or 0 have no proper tails. A word is \textit{Lyndon} if and only if it is lexicographically smaller than each of its proper
tails. For instance, $(3),(1,2,1,3)$ are Lyndon, but $(2,1,3)$ is not
Lyndon.
\end{defn}
We call a composition $\al=(a_1,\dots,a_m)$ is \textit{elementary} if
the greatest common divisor of its parts is 1, i.e. $\mb{gcd}\{a_1,\dots,a_m\}=1$. Denote by $\ly$ the set of Lyndon words over $\bN$ and $\el$ its subset of elementary ones. Let $\ly_{\od}:=\ly\cap\mc_{\od},\,
\el_{\od}:=\el\cap\mc_{\od}$. The \textit{Chen-Fox-Lyndon (CFL) factorization theorem} (see \cite[Theorem 6.5.5]{HNK}) is as follows.
\begin{theorem}
For each word $\al$, there is a unique concatenation factorization into nonincreasing Lyndon words
\begin{align*}
&\al=\ga_1^{*r_1}*\ga_2^{*r_2}*\cdots*\ga_k^{*r_k},\,\ga_i\in\ly,\\
&\ga_1>_{\lex}\ga_2>_{\lex}\cdots>_{\lex}\ga_k.
\end{align*}
\end{theorem}

We also need the following useful notion.
\begin{defn}
There is another total order on words over $\bN$, called the \textit{wll-ordering} and denoted $\leq_{\wll}$, where the acronym `wll' stands for `weight first, then length, then lexicographic'. Thus for example
\[(5)>_{\wll}(1,1,2)>_{\wll}(2,2)>_{\wll}(1,3).\]
\end{defn}

By Theorem \ref{qf}, one can consider QSym as a $\La$-module defined as in \eqref{ac}. Hence, for any composition $\al$, we write
\[p_n(\al):=p_n(M_\al)=\Psi^n(M_\al)\]
and
\beq\label{el} e_n(\al):=e_n(M_\al)=e_n(p_1(\al),\dots,p_n(\al)),\,n\in\bN.\eeq
Then by definition, $e_n(\al)=\la^n(M_\al)\in\mb{QSym}$.
Now we recall the well-known result that QSym is free commutative, first proved rigorously by Hazewinkel in \cite{Haz}.
\begin{theorem}[{\cite[Th. 6.7.5]{HNK}}]\label{fre}
$\{e_n(\al):\al\in\el,\,n\in\bN\}$ forms a
free commutative polynomial basis for QSym over the integers.
\end{theorem}

Via the lambda ring structure given in Theorem \ref{la}, PQSym can also serves as a $\La$-module defined as in \eqref{ac}. Here we just restrict it as a $\Ga$-module. For any $\al\in\mc_{\od}$, write
$p'_n(\al):=\Phi^n(L_\al)$ and
\beq\label{ql} q_n(\al):=q_n(L_{\al})=q_n(p'_1(\al),p'_2(\al),\dots),\,n\in\bN.\eeq
Note that $q_n(\al)$ can be obtained from the polynomial expression of $q_n$ in terms of the odd power sums by substituting $\Phi^n(L_\al)$ with these $p_n$. It is clear that $q_n(\al)=\tilde{\la}^n(L_\al)\in\mb{PQSym}$. In particular, $q_n((1))=q_n,\,n\in\bN$ by \eqref{qa}.

\begin{lem}\label{wl}
If $\al$ is a Lyndon odd composition, then
\[q_n(\al)=L_{\al^{*n}}+(\mb{wll-smaller than }\al^{*n}),\]
where (wll-smaller than $\al^{*n}$) stands
for a $\bZ$-linear combination of monomial peak functions whose indexes are wll-smaller than $\al^{*n}$.
\end{lem}
\begin{proof}
By formula \eqref{mo}, we deduce that
\beq\label{pm} L_\al L_\be=\sum_{\ga\notin\mc_{\bs{e}}}c_\ga L_{\vt(\ga)},\,\al,\be\in\mc_{\od},\eeq
when $\al\bowtie\be=\sum_{\ga\in\mc}c_\ga \ga$.
It is clear that $\vt(\ga)\leq_{\wll}\ga$ for any $\ga\notin\mc_{\bs{e}}$.

On the other hand, from the determinant expression \eqref{lps}, one can see that
\[n!q_n(\al)=L_\al^n+(\mb{monomials of length }\leq(n-1)\mb{ in the }p'_i(\al)).\]
Furthermore, all terms above are of equal weight. Since
$\al$ is Lyndon, the `length first-lexicographic thereafter' largest term in its $n$-th quasi-shuffle power $\al^{\bowtie n}$ is the concatenation power $\al^{*n}$ with the coefficient $n!$ (see \cite[Theorem 6.5.8]{HNK}). Hence, by formula \eqref{pm}, we get that
\[q_n(\al)=L_{\al^{*n}}+\sum_{\be<_{\wll}\al^{*n}}k_\be L_\be,\,k_\be\in\bQ.\]
Since $q_n(\al)\in\mb{PQSym}$ by \eqref{com}, all the coefficients $k_\be$ are indeed integers.
\end{proof}

Now we are in the position to state our main result, a structure theorem of PQSym.
\begin{theorem}\label{fr}
(i) The ring PQSym has a polynomial generating set \[\{q_n(\al):\al\in\el_{\od},\,n\in\bN\}\]
subject to the following complete relations:
\beq\label{re}
\sum_{i=0}^n(-1)^iq_i(\al)q_{n-i}(\al)=0,\,\al\in\el_{\od},n\in\bN.\eeq
In particular,
\beq\label{zb}\left\{\prod_{\al\in\el_{\od}}q_{\la_{\al}}(\al):\mb{ with finitely many nonempty strict partitions }\la_\al\right\}\eeq
is a $\bZ$-basis of PQSym, where  $q_\la(\al):=q_{\la_1}(\al)q_{\la_2}(\al)\cdots$ for any partition $\la$.

\noindent
(ii) The ring $\mb{PQSym}_\bQ$ has a free commutative polynomial basis \beq\label{ba}\{q_n(\al):\al\in\el_{\od},\,n\in\bN_{\od}\}.\eeq
Among this basis are the symmetric functions $q_n\,(n\in\bN_{\od})$ as a polynomial basis of $\Ga_\bQ$.
\end{theorem}
\begin{proof}
For (i) we first need to prove that every $\bZ$-basis
element $L_\be$ of PQSym can be written as a polynomial
in the $q_n(\al),\al\in\el_{\od},n\in\bN$. Let $A$ be the subring of PQSym generated by these $q_n(\al)$'s. To start with, let $\be=(b_1,\dots,b_r)$ be a Lyndon odd composition. Then taking $\al=\be_{\bs{red}}:=(g^{-1}b_1,\dots,g^{-1}b_r)$ with $g=g(\be):=\mb{gcd}\{b_1,\dots,b_r\}$ and  using \eqref{po}, we have $L_\be=p'_g(\al)$, which is an integral polynomial in the $q_n(\al)$  by the determinant expression \eqref{pl}. Hence, $L_\be\in A$ when $\be$ is Lyndon. We now proceed with induction for the wll-ordering. For each separate weight the induction starts, because compositions of length 1 (including the case of weight 1) are Lyndon. So let $\be$ be a composition of weight $\geq$ 2 and length $\geq$ 2. By the CFL factorization theorem,
\beq\label{fac}\begin{split}
&\be=\be_1^{*r_1}*\cdots*\be_k^{*r_k},\,\be_i\in\ly_{\od},\\
&\be_1>_{\lex}\be_2>_{\lex}\cdots>_{\lex}\be_k.
\end{split}\eeq

If $k\geq2$, take $\be'=\be_1^{*r_1}$ and $\be''$ as the corresponding tail so that $\be=\be'*\be''$. Then
\[L_{\be'}L_{\be''}=L_{\be'*\be''}+(\mb{wll-smaller than }\be)
=L_\be+(\mb{wll-smaller than }\be),\]
and with induction it follows that $\be\in A$.
There remains the case that $k=1$ in the CFL factorization \eqref{fac}. In this case take
$\al=(\be_1)_{\bs{red}},\,g=g(\be_1)$ and observe that by Lemma \ref{wl},
\[L_\be=q_n(p'_g(\al))+(\mb{wll-smaller than }\be).\]
On the other hand, by formula \eqref{co}
\[q_n(p'_g(\al))=(q_n\circ p_g)(\al),\]
where $q_n\circ p_g\in\Ga$ is some polynomial with integer coefficients in the $q_j$, and hence
$(q_n\circ p_g)(\al)$ is a polynomial with integer coefficients in the $q_j(\al)$. With induction this finishes the proof of generation.
Meanwhile, since $q_n(\al)=\tilde{\la}^n(L_\al)$, we know that the generators $q_n(\al)$ satisfy relations \eqref{re} by \eqref{eul}. It remains to be seen whether these relations are complete. We prove it by a counting argument, which also implies the statement in (ii) simultaneously.

Now consider the free commutative ring
\[Y:=\bZ[Y_n(\al):\al\in\el_{\od},n\in\bN_{\od}]\]
and the ring homomorphism
\[\rho:Y\rw\mb{PQSym},\,Y_n(\al)\mapsto q_n(\al).\]
According to relations \eqref{re}, all the $q_m(\al)$, with $m$ even, can be written as a linear combination of $q_n(\al)\,(n\in\bN_{\od})$ with rational coefficients.
Moreover, we have shown that $q_n(\al)\,(\al\in\el_{\od},n\in\bN)$ can generate the whole $\mb{PQSym}$, thus the morphism $\rho$ should be surjective when the base ring changes to $\bQ$. Also, giving $Y_n(\al)$ a weight $n\mb{wt}(\al)$, it is homogeneous. Note that there is bijection between
\[\{Y_i(\al):\al\in\el_{\od},i\in\bN_{\od},i\mb{wt}(\al)=n\}\mb{ and }\{\be\in\ly_{\od}:\mb{wt}(\be)=n\}.\]
Indeed, given $\be\in\ly_{\od}$, take $g=g(\be),\al=\be_{\bs{red}}$, then $\be\mapsto Y_g(\al)$ provides the bijection. Hence,
by the CFL factorization theorem, the rank of the weight $n$ component, denoted $Y(n)$, of $Y$ is equal to $|\mc_{\od}(n)|=f_{n-1}$, coinciding with the rank of the homogeneous $\mb{PQSym}_n$. It means that for any $n\in\bN$, the image of the restriction $\rho|_{Y(n)}$ is a proper abelian subgroup of $\mb{PQSym}_n$ of the same rank, thus the base-changed morphism $\rho_\bQ:Y_\bQ\rw\mb{PQSym}_\bQ$ is an isomorphism. Meanwhile, the generating set contains $q_n((1))=q_n,\,n\in\bN_{\od}$, so the proof for (ii) is completed.

On the other hand, it is well-known that the number of odd partitions of $n$ is equal to that of strict partitions of $n$, by the following identity of generating functions,
\[\sum_{\la\bs{ odd}}t^{|\la|}=\prod_{r\geq1}\dfrac{1}{1-t^{2r-1}}
=\prod_{r\geq1}\dfrac{1-t^{2r}}{1-t^r}=\prod_{r\geq1}(1+t^r)
=\sum_{\la\bs{ strict}}t^{|\la|}.\]
Hence, the cardinality
\[|\{q_\la(\al):\la\mb{ strict partition},|\la|\mb{wt}(\al)=n\}|=|\{q_\la(\al):\la\mb{ odd partition},|\la|\mb{wt}(\al)=n\}|\]
for any $\al\in\el_{\od}$. Also by relations \eqref{re}, it is clear that the elements in \eqref{zb} span PQSym, thus form a $\bZ$-basis of it, just as
\[\left\{\prod_{\al\in\el_{\od}}q_{\la_{\al}}(\al):\mb{with finitely many nonempty odd partitions }\la_\al\right\}\]
being a $\bQ$-basis of $\mb{PQSym}_\bQ$ by (ii). Finally, relations \eqref{re} are complete.
\end{proof}

\begin{cor}
For each $\al\in\el_{\od}$, the $q_n(\al)\,(n\in\bN)$ generate a lambda subring of PQSym, denoted by $\Ga_\al$. And $\Ga_\al$ is isomorphic to $\Ga$ by identifying $q_n(\al)$ with $q_n$. Hence,
\beq\label{ot}\mb{PQSym}\cong\bigotimes_{\al\in\el_{\od}}\Ga_\al,\eeq
as a tensor product of infinitely many copies of the $\la$-ring $\Ga$, one for each $\al\in\el_{\od}$. In particular, PQSym is free over $\Ga$
as $\Ga=\Ga_{(1)}$.
\end{cor}

For each $\al\in\el$, let $\La_\al$ be the lambda subring of QSym generated by $e_n(\al)\,(n\in\bN)$. By Theorem \ref{fre}, we know that
\beq\mb{QSym}\cong\bigotimes_{\al\in\el}\La_\al.\eeq
Since $e_n(\al)=\la^n(M_\al),\,q_n=\tilde{\la}^n(L_\al)$, we have
\[\vt(e_n(\al))=\begin{cases}
q_n(\al),&\al\in\el_{\od}\\
0,&\mb{otherwise}
\end{cases},\mb{ i.e. }\vt(\La_\al)=\begin{cases}
\Ga_\al,&\al\in\el_{\od}\\
0,&\mb{otherwise}
\end{cases},\]
for any $\al\in\el$, by \eqref{po} and \eqref{com}.

\begin{rem}
In \cite[Th. 4.1, Cor. 4.2]{Hsi}, Hsiao proved that the graded Hopf dual of $\mb{PQSym}_\bQ$ is a concatenation Hopf algebra over $\bQ$ by finding a free primitive generating set $\{\ve^*_n\}_{n\in\bN_{\od}}$, and $\mb{PQSym}_\bQ$ is free commutative with a generating set $\{\tau_\al\}_{\al\in\ly_{\od}}$ containing all odd power sums $p_n\,(n\in\bN_{\od})$. Here we directly find another polynomial basis \eqref{ba} of $\mb{PQSym}_\bQ$ containing all $q_n\,(n\in\bN_{\od})$ instead. Moreover, our polynomial basis in fact lies in PQSym, while the basis of Hsiao fails.

On the other hand, we have proved that PQSym is free over $\Ga$ in \cite[\S 4.2]{Li} via the terminology of Heisenberg doubles.
But a definite answer for such freeness can hardly be found in other references, thus we give an intrinsic proof here. Meanwhile,
it drives us to consider a more interesting problem, that is, to find a polynomial basis for PQSym over the integers, as the peak version of the Ditters conjecture. Unfortunately, there maybe even not exists any polynomial basis for its subring $\Ga$, by contrast with $\La$, which is the key to solve the problem by \eqref{ot}.
\end{rem}

\centerline{\bf Acknowledgments}
The author would like to thank the anonymous referee who reviewed \cite{Li} for valuable comments and the partial support of NSFC (Grant No. 11501214) for this work.

\bigskip
\bibliographystyle{amsalpha}

\end{document}